\documentclass[a4paper,12pt]{article}
\makeatletter
\@addtoreset{footnote}{page}
\makeatother

\usepackage{enumitem}

\usepackage{amsthm}
\usepackage{amsmath,amssymb,latexsym,amsfonts,mathrsfs}

\renewcommand{\Re}{\mathop{\rm Re}}

\newcommand{\eps}{\ensuremath{\varepsilon}}

\renewcommand{\tilde}{\widetilde}

\newcommand{\bC}{\ensuremath{\mathbb{C}}}

\newcommand{\bE}{\ensuremath{\mathbb{E}}}

\newcommand{\bP}{\ensuremath{\mathbb{P}}}

\newcommand{\bR}{\ensuremath{\mathbb{R}}}

\newcommand{\cE}{\ensuremath{\mathcal{E}}}
\newcommand{\cF}{\ensuremath{\mathcal{F}}}

\newcommand{\cP}{\ensuremath{\mathcal{P}}}

\theoremstyle{plain}
\newtheorem{Thm}{Theorem}[section]

\newtheorem{Lem}[Thm]{Lemma}
\newtheorem{Prop}[Thm]{Proposition}
\newtheorem{Cor}[Thm]{Corollary}

\theoremstyle{definition}

\newtheorem{Rem}[Thm]{Remark}

\numberwithin{equation}{section}

\makeatletter
\renewcommand\section{\@startsection {section}{1}{\z@}%
                                   {-3.5ex \@plus -1ex \@minus -.2ex}%
                                   {2.3ex \@plus.2ex}%
                                   {\normalfont\large\bf}}
\makeatother

\makeatletter
\renewcommand\subsection{\@startsection {subsection}{1}{\z@}%
                                   {-3.5ex \@plus -1ex \@minus -.2ex}%
                                   {2.3ex \@plus.2ex}%
                                   {\normalfont\normalsize\bf}}
\makeatother

\begin{document}
\begin{center}
	{\Large \bf
		Existence of quasi-stationary distributions for spectrally positive L\'evy processes on the half-line
	}
\end{center}
\begin{center}
	Kosuke Yamato (Kyoto University)
\end{center}
\begin{center}
	{\small \today}
\end{center}

\begin{abstract}
	For spectrally positive L\'evy processes killed on exiting the half-line,
	existence of a quasi-stationary distribution is characterized by the exponential integrability of the exit time, the Laplace exponent and the non-negativity of the scale functions.
	It is proven that if there is a quasi-stationary distribution, there are necessarily infinitely many ones and the set of quasi-stationary distributions is characterized.
	A sufficient condition for the minimal quasi-stationary distribution to be the Yaglom limit is given.
\end{abstract}

%%%%% text %%%%%

\section{Introduction}\label{section:intro}

For a stochastic process $\{ X_{t} \}_{t \geq 0} $ with the lifetime $\zeta$,
an initial distribution $\nu$ is called \textit{quasi-stationary distribution} when the distribution of $X_{t}$ starting from $\nu$ conditioned to survive at time $t$ is invariant in $t$,
that is, the following holds: 
\begin{align}
	\bP_{\nu}[X_{t} \in dx \mid \zeta > t] = \nu(dx) \quad (t > 0), \label{eq25}
\end{align} 
where $\bP_{\nu}$ denotes the underlying probability distribution of $X$ with the initial distribution $\nu$.
In the present paper, we study existence of a quasi-stationary distribution for spectrally positive L\'evy processes, that is, L\'evy processes without negative jumps or monotone paths, when they are killed at the exit time of the half-line.
More precisely, we consider existence of the probability distribution $\nu$ on $[0,\infty)$ satisfying \eqref{eq25} for $X$ being a spectrally positive L\'evy process and the lifetime $\zeta = \inf \{ t > 0 \mid X_{t} < 0 \}$.
We give a necessary and sufficient condition for the existence and characterize the set of quasi-stationary distributions.
Our approach is quite straightforward and mainly uses the \textit{scale function}, which is a fundamental tool in the study of spectrally one-sided L\'evy processes (see, e.g., \cite[Chapter 7]{BertoinLevy}, \cite{LevyMattersII} and \cite[Chapter 8]{KyprianouText}), and the related formulas for the exit time and potential density, which we will recall in Section 2.

% Let $X$ be a real-valued \textit{spectrally positive L\'evy process}, that is, the process $X$ is a L\'evy process without negative jumps or monotone paths.
% Suppose the process $X$ is killed on exiting the half-line $[0,\infty)$ and consider existence of an initial distribution $\nu$ which is invariant in time when conditioned to survive at each time $t > 0$:
% \begin{align}
% 	\bP_{\nu}[X_{t} \in dx \mid \tau^{-}_{0} > t] = \nu(dx) \quad (t > 0). \label{eq25}
% \end{align}
% where $\tau^{-}_{0} := \inf \{ t > 0 \mid X_{t} < 0 \}$ is the exit time of the half-line and
% $\bP_{\nu}$ denotes the underlying probability measure of $X$ with the initial distribution $\nu$.
% We call a probability distribution $\nu$ satisfying \eqref{eq25} \textit{quasi-stationary distribution}.
% In the present paper, we study existence of a quasi-stationary distribution and characterization of the set of quasi-stationary distributions.
% Our main results show that for any spectrally positive L\'evy process, either of the following two cases occurs: (i) there are no quasi-stationary distributions, (ii) there are infinitely many quasi-stationary distributions.
% We also give necessary and sufficient conditions for the existence in terms of the exponential integrability of the exit time $\tau^{-}_{0}$, the Laplace exponent and the scale functions.

On existence of a quasi-stationary distribution for spectrally positive L\'evy processes,
Bertoin \cite{BertoinQSD} has studied the case when killing occurs at the exit time of a finite interval $[0,a] \ (a \in (0,\infty))$ and showed under the assumption of the absolute continuity of the transition probability that there exists a unique quasi-stationary distribution $\nu$ such that
\begin{align}
	\bP_{\nu}[X_{t} \in dx \mid \tau^{(a)} > t] = \nu(dx) \quad (x \in [0,a]), \label{}
\end{align}
where $\tau^{(a)} := \inf \{ t > 0 \mid X_{t} \not\in [0,a] \}$, and it is represented by 
\begin{align}
	\nu(dx) := C W^{(-\lambda_{0})}(x)dx, \label{}
\end{align}
where $W^{(-\lambda_{0})}$ is the $(-\lambda_{0})$-scale function, which will be recalled in Theorem \ref{thm:scaleFunc} and Proposition \ref{prop:extScale},
$\lambda_{0} := \inf \{ \lambda \geq 0 \mid W^{(-\lambda)}(a) = 0 \}$ and a normalization constant $C > 0$.
He also showed the quasi-stationary distribution $\nu$ is the \textit{Yaglom limit}:
\begin{align}
	\bP_{x}[X_{t} \in dy \mid \tau^{(a)} > t] \xrightarrow[t \to \infty]{} \nu(dy) \quad (x \in (0,a)).\label{}
\end{align}
Here and hereafter, the convergence of distributions is in the sense of the weak convergence.
In the proof, he used the fact that the representation of the potential density by the scale function,
which we will recall in Theorem \ref{thm:potentialDensity}, is analytically extended to the meromorphic function on the complex plane and gave an elegant way combining a Tauberian theorem and the $R$-theory for general Markov processes developed by Tuominen and Tweedie \cite{TuominenTweedie}.
Although the analytic extension of the potential density is also important in our analysis,
we cannot appeal to the $R$-theory because in our situation the \textit{$\lambda_{0}$-recurrence}, which is an essential condition to apply the $R$-theory, does not hold as we will see in Remark \ref{rem:extOfPhi}.

Kyprianou and Palmowski \cite{KyprianouPalmowski} have studied existence of a quasi-stationary distribution for L\'evy processes killed on exiting the half-line $[0,\infty)$.
They gave two general classes of L\'evy processes which are not necessarily spectrally one-sided, and showed existence of a quasi-stationary distribution and the Yaglom limit for the processes in the classes.
Though their results are widely applicable,
their assumptions are a little strong when we focus on existence of a quasi-stationary distribution and restrict our attention to spectrally positive L\'evy processes.
In addition, their result only shows existence of the minimal quasi-stationary distribution in a certain stochastic order (for details see Corollary \ref{cor:QSDsOrdered} and the preceding description).
As we will show in Theorem \ref{thm:charOfQSD}, for spectrally positive L\'evy processes, if there is a quasi-stationary distribution, there are necessarily infinitely many ones.

\subsection{Main results}

For a spectrally positive L\'evy process $X$,
we denote the underlying probability measure starting from $x \in \bR$ by $\bP_{x}$ and we simply write $\bP := \bP_{0}$.
The process $X$ is characterized by the Laplace exponent
\begin{align}
	\psi(\beta) &:= \log \bE[\mathrm{e}^{-\beta X_{1}}] \label{} \\
	&= -a\beta + \frac{1}{2}\sigma^{2}\beta^2 + \int_{0}^{\infty}(\mathrm{e}^{-\beta x} - 1 + \beta x 1\{x<1\})\Pi(dx) \quad (\beta \in \bR), \label{LevyKhintchine}
\end{align}
where $a \in \bR$ is a constant, $\sigma \geq 0$ is a Gaussian coefficient and $\Pi$ is a L\'evy measure, that is, a Radon measure on $(0,\infty)$ with $\int_{0}^{\infty}(1 \wedge x^{2})\Pi(dx) < \infty$.
Note that $|\psi(\beta)| < \infty$ for $\beta \geq 0$, while $-\infty < \psi(\beta) \leq \infty$ for $\beta < 0$.
Since the process $X$ is not a subordinator nor a constant drift, we have the following three cases (see e.g., \cite[Lemma 2.14]{KyprianouText}):
\begin{align}
	\begin{aligned}
		\text{(i)} \quad
		&\text{$\sigma > 0$.
		} \\
		\text{(ii)} \quad
		&\text{$ \sigma = 0$\quad and \quad$\int_{0}^{1}x\Pi(dx) = \infty$.
		} \\
		\text{(iii)} \quad 
		&\text{$\sigma = 0$\quad and \quad$a < \int_{0}^{1-} x\Pi(dx) < \infty$.
		}
	\end{aligned}
	\label{eq22}
\end{align}
Define
\begin{align}
	\lambda_{0} := \sup \{ \lambda \geq 0 \mid \bE_{x}[\mathrm{e}^{\lambda \tau^{-}_{0}}, \tau_{0} < \infty] < \infty \ \text{for some } x > 0 \}, \label{eq08}
\end{align}
where $\tau_{0}^{-} := \inf \{ t > 0 \mid X_{t} < 0 \}$.
Note that $\lambda_{0}$ is always finite since $\sigma > 0$ or $\Pi(0,\infty) > 0$.
Also note that, for any $\lambda \geq 0$ by the strong Markov property and the space-homogeneity, if  $\bE_{x}[\mathrm{e}^{\lambda\tau^{-}_{0}}, \tau_{0} < \infty] < \infty$ for some $x > 0$, it also holds for every $x > 0$.

Let us give a necessary and sufficient condition for existence of a quasi-stationary distribution.
The $q$-scale function $W^{(q)}(x) \ (q \in \bC,\ x \in \bR)$ will be recalled in Section \ref{section:preliminaries}.
Note that for $q \in \bR$, the function $W^{(q)}$ is real-valued.
The positivity and integrability of $W^{(-\lambda)} \ (\lambda \in (0,\lambda_{0}])$ will be proven in Lemmas \ref{lem:non-negativityOfScale} and \ref{lem:non-negativeImpliyIntegrable}, respectively.

\begin{Thm} \label{thm:charOfQSD}
	For a spectrally positive L\'evy process $X$, a quasi-stationary distribution exists if and only if $\lambda_{0} > 0$.
	If $\lambda_{0} > 0$, the set of quasi-stationary distributions is $\{ \nu_{\lambda} \}_{0 < \lambda \leq \lambda_{0}}$, where 
	\begin{align}
		\nu_{\lambda}(dx) := \lambda W^{(-\lambda)}(x)dx, \label{}
	\end{align}
	and it holds
	\begin{align}
		\bP_{\nu_{\lambda}}[X_{t} \in dx, \tau_{0} > t] = \mathrm{e}^{-\lambda t}\nu(dx). \label{}
	\end{align}
\end{Thm}

\begin{Rem}
	On existence of a quasi-stationary distribution, there is a similarity between spectrally positive L\'evy processes and one-dimensional diffusions.
	Let $Y$ be a $\frac{d}{dm}\frac{d}{ds}$-diffusion $[0,\infty)$ eventually hitting $0$; $\bP_{x}^{Y}[\tau_{0}^{Y} < \infty] = 1 \ (x > 0)$, where $\bP^{Y}_{x}$ denotes the law of $Y$ starting from $x$ and $\tau^{Y}_{y} := \inf \{ t > 0 \mid Y_{t} = y \} \ (y \geq 0)$.
	For $q \in \bC$, let $u = \psi_{q}$ be the unique solution to the following equation:
	\begin{align}
		\frac{d}{dm}\frac{d^{+}}{ds} u(x) = q u(x), \quad u(0) = 0, \quad \frac{d^{+}}{ds}u(0) = 1 \quad (x \geq 0), \label{}
	\end{align}
	where $\frac{d^{+}}{ds}$ denotes the right-differential operator w.r.t.\ the function $s$.
	For $q \geq 0$, set $g_{q}(x) := \bE^{Y}_{x}[\mathrm{e}^{-q\tau^{Y}_{0}}] \ (x \geq 0)$.
	Then $u = g_{q}(x)$ is the unique non-negative decreasing solution to
	\begin{align}
		\frac{d}{dm}\frac{d^{+}}{ds}u(x) = qu(x), \quad u(0) = 1, \quad \lim_{x \to \infty}\frac{d^{+}}{ds}u(x) = 0 \quad (x \geq 0) \label{} 
	\end{align}
	(see \cite[Theorem 5.13.3]{Ito_essentials}).
	The exit time from an interval is characterized by these functions: for $0 \leq b < a$
	\begin{align}
		\bE^{Y}_{x}[\mathrm{e}^{-q\tau_{b}}, \tau_{b} < \tau_{a}] = \frac{\Phi^{(q)}(x,a)}{\Phi^{(q)}(b,a)} \quad (x \in [b,a], q \geq 0), \label{}
	\end{align}
	where
	\begin{align}
		\Phi^{(q)}(x,y) := \psi_{q}(y)g_{q}(x) - g_{q}(y)\psi_{q}(x) \label{}
	\end{align}
	(see \cite[Theorem 5.15.1]{Ito_essentials}).
	As we will recall in Theorem \ref{thm:scaleFunc},
	the similar representation for the exit time holds for spectrally positive L\'evy processes by replacing $\Phi^{(q)}$ with $W^{(q)}$ though this similarity is classically known.
	Note that $u = \Phi^{(q)}(\cdot,y) \ (y > 0)$ is a unique non-negative decreasing solution to
	\begin{align}
		\frac{d}{dm}\frac{d^{+}}{ds}u(x) = qu(x), \quad u(0) = \psi_{q}(y), \quad u(y) = 0 \quad (x \in [0,y]). \label{}
	\end{align}
	In particular, it holds $\Phi^{(q)}(0,y) = \psi_{q}(y)$.

	It is known that for $a > 0$ there exists a unique quasi-stationary distribution $\nu$ when $Y$ is killed at $\tau_{0} \wedge \tau_{a}$ (see e.g., \cite[Theorem 3.1, Example 3.1]{Takeda:QSD}) and
	it is given by
	\begin{align}
		\nu(dx) = C \psi_{-\lambda_{0}^{Y}(a)}(x)dm(x) \quad (x \in [0,a]), \label{}
	\end{align}
	where $\lambda_{0}^{Y}(a) := \inf \{\lambda \geq 0 \mid \psi_{-\lambda}(a) = 0\}$ and $C > 0$ is the normalizing constant.
	This result is similar to the result by Bertoin \cite{BertoinQSD} we have already seen.
	Next we consider the case $Y$ is killed at $\tau_{0}$.	
	It is known that there exists infinitely many quasi-stationary distributions if and only if the boundary $\infty$ is natural in the sense of Feller (see e.g., \cite[Section 5.11]{Ito_essentials}) and
	\begin{align}
		\lambda_{0}^{Y} := \sup \{ \lambda \geq 0 \mid \bE_{x}[\mathrm{e}^{\lambda \tau_{0}^{Y}}] < \infty \quad \text{for some $x > 0$} \} > 0 \label{}
	\end{align}
	(see e.g., \cite[Theorem 6.34]{Quasi-stationary_distributions}).
	In the case, the set of quasi-stationary distributions is given by $\{ \nu_{\lambda}^{Y} \}_{\lambda \in (0,\lambda_{0}^{Y}]}$ for
	\begin{align}
		\nu^{Y}_{\lambda}(dx) := \lambda \psi_{-\lambda}(x)dm(x) \quad (x \in [0,\infty)). \label{}
	\end{align}
	This result is similar to Theorem \ref{thm:charOfQSD}.
\end{Rem}

\begin{Rem}
	As we will see in Corollary \ref{cor:QSDsOrdered}, quasi-stationary distributions are totally ordered by two stochastic orders both related to the Laplace transform and the distribution $\nu_{\lambda_{0}}$ is the minimal element in the orders.
\end{Rem}

Let us characterize $\lambda_{0}$ in two ways: one is in terms of the scale functions and the other is in terms of the minimum value of the Laplace exponent.
Define
\begin{align}
	\theta_{0} := \sup \{ \theta \geq 0 \mid \psi(-\theta) \in \bR  \quad \text{and} \quad \psi^{+}(-\theta) \geq 0 \}, \label{thetaZero}
\end{align}
where we regard $\sup \emptyset = 0$ and $\psi^{+}$ denotes the right-derivative of $\psi$.
Note that $\theta_{0}$ is always finite.
Indeed, if $|\psi(-\theta)| < \infty$ for every $\theta > 0$, it follows $\lim_{\theta \to \infty}\psi(-\theta) = \infty$ from \eqref{eq22}, and since $\psi(\beta)$ is convex on $(-\infty,0)$ in the case, it follows $\psi'(-\theta) = 0$ for some $\theta \geq 0$.
Also note that $\psi(-\theta_{0}) \in \bR$ from \eqref{LevyKhintchine} and the monotone convergence theorem. 

\begin{Thm} \label{thm:spectralBottom}
	The following equalities hold:
	\begin{align}
		\lambda_{0} = \sup \left\{ \lambda \geq 0 ~\middle|~ W^{(-\lambda)}(x) > 0 \quad \text{for every $x > 0$} \right\} = -\psi(-\theta_{0}). \label{eq07}
	\end{align}
\end{Thm}

As a consequence of Theorem \ref{thm:spectralBottom}, we have the following corollary,
which gives a simple criterion for existence of quasi-stationary distributions.
From this corollary we see that the Gaussian coefficient does not affect existence of quasi-stationary distributions as long as \eqref{eq22} is satisfied.

\begin{Cor}
	For a spectrally positive L\'evy process $X$,
	a quasi-stationary distribution exists if and only if the following holds:
	\begin{align}
		\psi^{+}(0) > 0 \quad \text{and} \quad \int_{1}^{\infty}\mathrm{e}^{\delta x}\Pi(dx) < \infty \quad \text{for some } \delta > 0. \label{}
	\end{align}
\end{Cor}

\begin{proof}
	From Theorem \ref{thm:charOfQSD} and Theorem \ref{thm:spectralBottom}, existence of quasi-stationary distributions is equivalent to the positivity of $\theta_{0}$.
	It is well-known that the finiteness $|\psi(-\beta)|$ is equivalent to that of $\int_{1}^{\infty}\mathrm{e}^{\beta x}\Pi(dx)$ (see e.g., \cite[Theorem 3.6]{KyprianouText}). 
\end{proof}

We give a sufficient condition for the minimal quasi-stationary distribution $\nu_{\lambda_{0}}$ to be the Yaglom limit of every compactly supported initial distribution $\mu$ on $(0,\infty)$, that is, the following holds:
\begin{align}
	\bP_{\mu}[X_{t} \in dx \mid \tau^{-}_{0} > t] \xrightarrow[t \to \infty]{} \nu_{\lambda_{0}}(dx). \label{eq28}
\end{align}
Before stating the result, we recall the \textit{Esscher transform}, the exponential change of the measure 
\begin{align}
	\left.\frac{d\bP^{\theta}}{d\bP}\right|_{\cF_{t}} = \cE_{t}, \label{eq16}
\end{align}
by the exponential martingale $\cE_{t} := \mathrm{e}^{\theta X_{t} - \psi(-\theta)t}$ for $\theta \in \bR$ such that $\psi(-\theta) \in \bR$, where $\{\cF_{t}\}_{t \geq 0}$ is the natural filtration of $X$.
It is well-known that under $\bP^{\theta}$, the process $X$ is again a spectrally positive L\'evy process with the Laplace exponent $\psi_{\theta}(\beta) := \psi(\beta - \theta) - \psi(-\theta) \ (\beta \geq 0)$ (see e.g., \cite[p.82]{KyprianouText}).
The following corollary is just an application of \cite[Theorem 1]{KyprianouPalmowski}. Note that the assumption $X_{1}$ is non-lattice is not restrictive since every quasi-stationary distribution is absolutely continuous.

\begin{Cor} \label{cor:YaglomLimit}
	Let $\lambda_{0} > 0$.
	Suppose $X_{1}$ is non-lattice and either of the following conditions holds:
	\begin{enumerate}
		\item Under $\bP^{\theta_{0}}$, the process $X$ is in the domain of attraction of an $\alpha$-stable distribution for $1 < \alpha \leq 2$.
		\item The function $x \mapsto \int_{x}^{\infty}\mathrm{e}^{\theta_{0}y}\Pi(dy)$ is regularly varying at $\infty$ with exponent $-\beta$ for $\beta > 2$.
	\end{enumerate}
	Then the convergence \eqref{eq28} holds for every compactly supported distribution $\mu$ on $(0,\infty)$.
\end{Cor}

\begin{Rem}
	If $\psi(-\theta_{0} - \delta) \in \bR$ for some $\delta > 0$, the condition (i) holds.
	Indeed, it implies $\int_{0}^{\infty}x^{2}\mathrm{e}^{\theta_{0}x}\Pi(dx) < \infty$ and hence $\bE^{\theta_{0}}|X_{1}|^{2} < \infty$. By the central limit theorem, we see that $X$ satisfies the condition (i) for $\alpha = 2$.

\end{Rem}

\subsection*{Outline of the paper}

% In Section \ref{section:prevStudies}, we briefly review the previous studies related to the quasi-stationary distributions for spectrally positive L\'evy processes.
In Section \ref{section:preliminaries}, we will prepare notation and recall some known results.
In Section \ref{section:proofs}, we will prove Theorems \ref{thm:charOfQSD}, \ref{thm:spectralBottom} and Corollary \ref{cor:YaglomLimit}, showing auxiliary results.
In Section \ref{section:example}, we will show some examples for which we can compute the scale functions explicitly to some extent and give the quasi-stationary distributions.

\subsection*{Acknowledgement}

The author would like to thank Kouji Yano who read an early draft of the present paper and gave him valuable comments.
This work was supported by JSPS KAKENHI Grant Number JP21J11000, JSPS Open Partnership Joint Research
Projects Grant Number JPJSBP120209921 and the Research Institute for Mathematical Sciences,
an International Joint Usage/Research Center located in Kyoto University and, was carried out under the ISM Cooperative Research Program (2020-ISMCRP-5013).

\section{Preliminaries} \label{section:preliminaries}

Here we recall some basic properties for spectrally positive L\'evy processes.
Most results in this section can be seen in \cite[Chapter 8]{KyprianouText} and \cite[Chapter 7]{BertoinLevy}, and thus the proof is omitted in most cases.
In the sequel, the process $X$ always denotes a spectrally positive L\'evy process.

\subsection{Exit time}

Since it holds $\psi''(\beta) = \sigma^{2} + \int_{0}^{\infty}x^{2}\mathrm{e}^{-\beta x}\Pi(dx) > 0$,
the Laplace exponent $\psi(\beta)$ is strictly convex and it holds $\bE X_{1} = - \psi^{+}(0) \in (-\infty,\infty]$.
% At first, we consider when the process $X$ certainly exit the half-line.
% \begin{Prop}
% 	The Laplace exponent $\psi$ is strictly convex and 
% 	\begin{align}
% 		\bE X_{1} = - \psi'(0+) \in (-\infty,\infty]. \label{}
% 	\end{align}
% \end{Prop}
% \begin{proof}
% 	Strict convexity follows from
% 	\begin{align}
% 		\psi''(\beta) = \sigma^{2} + \int_{0}^{\infty}x^{2}\mathrm{e}^{-\beta x}\Pi(dx) > 0. \label{}
% 	\end{align}
% 	Since $\psi'(\beta)$ is strictly increasing, we have
% 	\begin{align}
% 		-\bE X_{1} = - a - \int_{1}^{\infty}x \Pi(dx) = \psi'(0+) \in [-\infty,\infty). \label{}
% 	\end{align}
% \end{proof}
Define the right inverse of $\psi$:
\begin{align}
	\Phi(q) := \sup \{ \beta \geq 0 \mid \psi(\beta) = q \} \quad (q \geq 0).
\end{align}
Since $X$ is not a subordinator, it holds $\psi(\infty) = \infty$.
Indeed, if $\sigma > 0$, it is obvious and if $\sigma = 0$, by Fatou's lemma $\liminf_{\beta \to \infty} \psi(\beta) / \beta \geq -a + \int_{0}^{1-} x \Pi(dx) > 0$.
From this, we also see $\Phi(\infty) = \infty$.
The following basic formula plays an important role in our analysis.

\begin{Thm}[{\cite[Theorem 3.12]{KyprianouText}}] \label{thm:LTofFHT}
	It holds
	\begin{align}
		\bE_{x}[\mathrm{e}^{-q \tau^{-}_{0}}, \tau_{0}^{-} < \infty ] = \mathrm{e}^{-x \Phi(q)} \quad (x \geq 0, \ q \geq 0). \label{eq12}
	\end{align}
\end{Thm}

From Theorem \ref{thm:LTofFHT}, we see the following equivalence for $x > 0$:
\begin{align}
	\bP_{x}[\tau^{-}_{0} < \infty] = 1 \quad \Leftrightarrow \quad \Phi(0) = 0 \quad \Leftrightarrow \quad \psi^{+}(0) \geq 0 \quad \Leftrightarrow \quad \bE X_{1} \leq 0. \label{eq31} 
\end{align}

\begin{Rem} \label{rem:FHU}
	From Theorem \ref{thm:LTofFHT}, we see the map $\cP[0,\infty) \ni \mu \mapsto \bP_{\mu}[\tau^{-}_{0} \in dt]$ is injective, where $\cP[0,\infty)$ denotes the set of probability distributions on $[0,\infty)$.
	Indeed, for every $\mu \in \cP[0,\infty)$, it holds
	\begin{align}
		\bE_{\mu}[\mathrm{e}^{-q\tau^{-}_{0}}, \tau_{0}^{-} < \infty] = \int \mathrm{e}^{-x\Phi(q)}\mu(dx), \label{}
	\end{align}
	and the RHS is the Laplace transform of $\mu$ evaluated at $\Phi(q)$.
	Since $\Phi$ is continuous and strictly increasing, the distribution $\mu$ is determined by the distribution $\bP_{\mu}[\tau^{-}_{0} \in dt]$.
	In \cite{QSD_unifying_approach} we called the injectivity the \textit{first hitting uniqueness}.
	By \cite[Theorem 1.1]{QSD_unifying_approach} (while it is shown for one-dimensional diffusions, the same proof works for much more general Markov processes),
	we see the convergence $\bP_{\mu}[X_{t} \in dx \mid \tau^{-}_{0} > t] \to \nu(dx) \ (t \to \infty)$ is reduced to the tail behavior of $\bP_{\mu}[\tau^{-}_{0} \in dt]$.
\end{Rem}

\subsection{Scale functions}

Define the first hitting times:
\begin{align}
	\tau^{+}_{x} := \inf \{ t > 0 \mid X_{t} > x \}, \quad
	\tau^{-}_{x} := \inf \{ t > 0 \mid X_{t} < x \} \quad (x \in \bR). 
\end{align}
The following gives the definition and properties of the $q$-scale function $W^{(q)}$ for $q \geq 0$.
\begin{Thm}[{\cite[Theorem 8.1]{KyprianouText}}] \label{thm:scaleFunc}
	For $q \geq 0$, there exists a function $W^{(q)}: \bR \to [0,\infty)$ (we write $W := W^{(0)}$) such that the following hold:
	\begin{enumerate}
		\item For $x < 0$, the function $W^{(q)}(x) = 0$, and $W^{(q)}$ is characterized on $[0,\infty)$ as a strictly increasing and continuous function whose Laplace transform satisfies
		\begin{align}
			\int_{0}^{\infty} \mathrm{e}^{-\beta x} W^{(q)}(x)dx = \frac{1}{\psi(\beta) - q} \quad (\beta > \Phi(q)). \label{}
		\end{align}
		\item For $0 \leq x \leq a$, it holds
		\begin{align}
			\bE_{x}[\mathrm{e}^{-q \tau^{-}_{0}}, \tau^{-}_{0} < \tau^{+}_{a} ] = \frac{W^{(q)}(a-x)}{W^{(q)}(a)}. \label{eq01}
		\end{align}
	\end{enumerate}
\end{Thm}

We derive an easy consequence of Theorem \ref{thm:scaleFunc} (ii) for later use.

\begin{Cor} \label{cor:ratioOfScale}
	It holds for $q \geq 0$
	\begin{align}
		\lim_{a \to \infty} \frac{W^{(q)}(a-x)}{W^{(q)}(a)} = \mathrm{e}^{-x \Phi(q)} \quad (x \geq 0). \label{eq02}
	\end{align}
\end{Cor}

\begin{proof}
	From \eqref{eq01} and Theorem \ref{thm:LTofFHT}, we have
	\begin{align}
		\mathrm{e}^{-x \Phi(q)} = \lim_{a \to \infty} \bE_{x}[\mathrm{e}^{-q \tau^{-}_{0}}, \tau^{-}_{0} < \tau^{+}_{a} ] = \lim_{a \to \infty} \frac{W^{(q)}(a-x)}{W^{(q)}(a)} \quad (q \geq 0). \label{}
	\end{align}
\end{proof}

Following \cite{BertoinQSD}, the $q$-scale function $q \mapsto W^{(q)}(x)$ can be analytically extended to the entire function for every fixed $x \geq 0$.
The proof of the following proposition can be seen in \cite[Lemma 8.3]{KyprianouText}.
The characterization of the scale function by a kind of renewal equation \eqref{charEq} is useful for our purpose.
We denote the convolution of functions $f$ and $g$ on $\bR$ by
\begin{align}
	f \ast g (x) := \int_{\bR}f(x-y)g(y)dy \quad (x \in \bR) \label{}
\end{align}
and denote the $n$-th convolution of $f$ by $f^{\ast 1} := f$ and $f^{\ast n}(x) := f^{\ast(n-1)} \ast f \ (n \geq 2)$.

\begin{Prop}[{\cite[Lemma 4]{BertoinQSD} and \cite[Lemma 8.3]{KyprianouText}}] \label{prop:extScale}
	For each $x \geq 0$, the function $q \mapsto W^{(q)}(x)$ may be analytically extended in $q$ to $\bC$.
	Moreover, the extension satisfies
	\begin{align}
		W^{(q)}(x) = \sum_{k \geq 0}q^{k} W^{\ast (k+1)}(x) \quad (q \in \bC, \ x \geq 0) \label{seriesExpOfScale}
	\end{align}
	and
	\begin{align}
		|W^{(q)}(x)| \leq W^{(|q|)}(x) \leq W(x) \mathrm{e}^{|q| \int_{0}^{x}W(y)dy}. \label{}
	\end{align}
	In addition, it holds
	\begin{align}
		\int_{0}^{\infty} \mathrm{e}^{-\beta x} W^{(q)}(x) dx = \frac{1}{\psi(\beta) - q} \quad (q \in \bC, \ \beta > \Phi(|q|)) \label{eq05}
	\end{align}
	and $f = W^{(q)}$ is the unique solution for the following equation for each $r \in \bC$:
	\begin{align}
		f(x) = W^{(r)}(x) + (q-r) W^{(r)} \ast f(x) \quad (x \geq 0) \label{charEq}
	\end{align}
	for a function $f$ with $\int_{0}^{\infty}\mathrm{e}^{-\beta x}|f(x)|dx < \infty$ for large $\beta > 0$.
\end{Prop}

% \begin{proof}
% 	At first, we show that the series
% 	\begin{align}
% 		\tilde{W}^{(q)}(x) := \sum_{k \geq 0}q^{k} W^{\ast (k+1)}(x) \quad (x \geq 0, \ q \in \bC) \label{}
% 	\end{align}
% 	is absolutely convergent.
% 	Since it holds
% 	\begin{align}
% 		W^{\ast 2}(x) = \int_{0}^{x} W(x-y)W(y)dy \leq W(x) G(x) \label{}
% 	\end{align}
% 	for $G(x) := \int_{0}^{x} W(y) dy$, we have inductively
% 	\begin{align}
% 		W^{\ast(k+1)}(x) &= \int_{0}^{x} W^{\ast k}(x-y) dG(y) \leq \frac{1}{k!}W(x)G(x)^{k}. \label{}
% 	\end{align}
% 	Thus, we have
% 	\begin{align}
% 		\sum_{k \geq 0}|q|^{k} W^{\ast (k+1)}(x) \leq W(x) \mathrm{e}^{|q| G(x)} < \infty. \label{}
% 	\end{align}
% 	Let $q \in \bC$ and $\beta > \Phi(|q|)$. Note that $\psi(\beta) > |q|$. Then it holds
% 	\begin{align}
% 		\int_{0}^{\infty}\mathrm{e}^{-\beta x} \tilde{W}^{(q)}(x) dx 
% 		&= \sum_{k \geq 0} q^{k} \int_{0}^{\infty} \mathrm{e}^{-\beta x} W^{\ast (k+1)}(x)dx \label{} \\
% 		&= \sum_{k \geq 0} q^{k} \frac{1}{\psi(\beta)^{k+1}} \label{} \\
% 		&= \frac{1}{\psi(\beta) - q}. \label{eq04}
% 	\end{align}
% 	Thus, from the uniqueness of the Laplace transform and Theorem \ref{thm:scaleFunc} (i), we have $\tilde{W}^{(q)}(x) = W^{(q)}(x) \ (x \geq 0)$ for $q \geq 0$.
% 	The fact that $f = W^{(q)}$ is the unique solution of \eqref{charEq} easily follows by the Laplace transform of both sides and \eqref{eq05}. 
% \end{proof}

\subsection{Potential density}

For $q \geq 0$, define the $q$-potential measure killed on exiting the interval $[0,a]$ by
\begin{align}
	U^{(q)}(a,x,dy) := \int_{0}^{\infty} \mathrm{e}^{-q t} \bP_{x}[X_{t} \in dy, \tau^{+}_{a} \wedge \tau^{-}_{0} > t]dt \quad (x \in [0,a]). \label{}
\end{align}
The $q$-potential density exists and can be represented by the scale function.
\begin{Thm}[{\cite[Theorem 8.7]{KyprianouText}}] \label{thm:potentialDensity}
	For $q \geq 0$, the $q$-potential measure has a density $u^{(q)}(a,x,y)$ w.r.t. the Lebesgue measure:
	\begin{align}
		U^{(q)}(a,x,dy) = u^{(q)}(a,x,y)dy \quad (x,y \in [0,a]) \label{}
	\end{align}
	for
	\begin{align}
		u^{(q)}(a,x,y) := \frac{W^{(q)}(a-x) W^{(q)}(y)}{W^{(q)}(a)} - W^{(q)}(y-x). \label{potentialDensityTwoSide}
	\end{align}
\end{Thm}
By taking the limit as $a \to \infty$ in \eqref{potentialDensityTwoSide}, from Corollary \ref{cor:ratioOfScale} we obtain the potential density killed on exiting the half-line.

\begin{Cor}[{\cite[Corollary 8.8]{KyprianouText}}] \label{cor:potentialDensity}
	For $q \geq 0$, the $q$-potential measure killed on exiting $[0,\infty)$ has a density w.r.t. the Lebesgue measure:
	\begin{align}
		\int_{0}^{\infty} \mathrm{e}^{-q t} \bP_{x}[X_{t} \in dy, \tau^{-}_{0} > t]dt = u^{(q)}(x,y)dy  \quad (x,y \geq 0)\label{}
	\end{align}
	for
	\begin{align}
		u^{(q)}(x,y) := \mathrm{e}^{-x \Phi(q)}W^{(q)}(y) - W^{(q)}(y-x). \label{eq29}
	\end{align}
\end{Cor}

\section{Proof of the main results} \label{section:proofs}

To prove Theorem \ref{thm:charOfQSD}, we prepare the following four auxiliary lemmas.

First, we show that if there exists a quasi-stationary distribution, its density is given by a scale function.
Note that for a quasi-stationary distribution $\nu$ the lifetime distribution $\bP_{\nu}[\tau^{-}_{0} \in dt]$ is exponentially distributed since it holds from the Markov property $\bP_{\nu}[\tau^{-}_{0} > t+s] = \bP_{\nu}[\tau^{-}_{0} > t]\bP_{\nu}[\tau^{-}_{0} > s] \ (t,s > 0)$.

\begin{Lem}\label{lem:charOfQSD1}
	Assume $\nu$ is a quasi-stationary distribution for $X$ such that $\bP_{\nu}[\tau^{-}_{0} > t] = \mathrm{e}^{-\lambda t} \ (t > 0)$ for some $\lambda > 0$. 
	Then $\nu$ has a density w.r.t. the Lebesgue measure and its density can be given by $\lambda W^{(-\lambda)}(x)$; $\nu(dx) = \lambda W^{(-\lambda)}(x)dx$.
\end{Lem}

\begin{proof}
	From the definition of quasi-stationary distributions and Corollary \ref{cor:potentialDensity} for $q = 0$,
	it holds
	\begin{align}
		\nu(dy) = \lambda\int_{0}^{\infty} \bP_{\nu}[X_{t} \in dy, \tau^{-}_{0} > t]dt 
		= \lambda \left(\int_{0}^{\infty}u^{(0)}(x,y)\nu(dx)\right)dy. \label{}
	\end{align}
	Thus, $\nu$ is absolutely continuous. We denote the density by $\rho$.
	Again from Corollary \ref{cor:potentialDensity}, we have
	\begin{align}
		\rho(x) &= \lambda \int_{0}^{\infty}(W(x) - W(x-y))\rho(y)dy  \label{} \\
		&= \lambda W(x) - \lambda W \ast \rho (x). \label{} 
	\end{align}
	It follows $\rho(x) = \lambda W^{(-\lambda)}(x)$ a.e. from \eqref{charEq}.
\end{proof}

Non-negativity of $W^{(-\lambda)}$ implies the integrability.

\begin{Lem} \label{lem:non-negativeImpliyIntegrable}
	Let $\lambda > 0$. If the function $W^{(-\lambda)}$ is non-negative on $(0,\infty)$,
	the function $W^{(-\lambda)}(x)$ is integrable on $(0,\infty)$ and it holds
	\begin{align}
		1 = \lambda \int_{0}^{\infty} W^{(-\lambda)}(y)dy. \label{eq06}
	\end{align} 
\end{Lem}

\begin{proof}
	From \eqref{eq01} for $q = 0$ and the equation \eqref{charEq} for $q=-\lambda, \ r = 0$, we have
	\begin{align}
		\frac{W^{(-\lambda)}(x)}{W(x)} = 1 - \lambda \int_{0}^{x} \bP_{y}[\tau^{-}_{0} < \tau^{+}_{x}] W^{(-\lambda)}(y)dy. \label{eq11}
	\end{align}
	By the monotone convergence theorem, we see
	\begin{align}
		0 \leq \lim_{x \to \infty} \frac{W^{(-\lambda)}(x)}{W(x)} = 1 - \lambda \int_{0}^{\infty}W^{(-\lambda)}(x)dx, \label{}
	\end{align}
	and we obtain the integrability of $W^{(-\lambda)}$.
	If $\lim_{x \to \infty} W^{(-\lambda)}(x) / W(x) =: \delta > 0$,
	it follows $W^{(-\lambda)}(x) > (\delta/2)W(x)$ for large $x$,
	and it implies that $W$ is integrable on $(0,\infty)$.
	It is, however, impossible since $W$ is increasing.
	Hence, we obtain \eqref{eq06}.
\end{proof}

To show the positivity of $W^{(-\lambda)}(x)$, we need the exponential integrability of the exit time.
We owe the proof of the following proposition largely to \cite[Theorem 2]{BertoinQSD}.

\begin{Lem} \label{lem:non-negativityOfScale}
	Assume $\lambda_{0} > 0$.
	Then for $0 < \lambda \leq \lambda_{0}$, it holds $W^{(-\lambda)}(x) > 0 \ (x > 0)$.
\end{Lem}

\begin{proof}
	Fix $0 < \lambda < \lambda_{0}$.
	We prove by contradiction.
	Suppose $W^{(-\lambda)}(x) = 0$ for some $x > 0$.
	Let $x_{0}$ be infimum of the root of $W^{(-\lambda)}(x)$ on $x > 0$.
	Since $W^{(-\lambda)}$ is increasing near $x = 0$, it holds $x_{0} > 0$.
	From the analytic extension of \eqref{potentialDensityTwoSide}, we have for $q \in \bC$ such that $W^{(q)}(x_{0}) \neq 0$ and $\Re q > -\lambda_{0}$ and for $0 < x < x_{0}$
	\begin{align}
		&\int_{0}^{\infty}\mathrm{e}^{-q t} \bP_{x}[\tau^{+}_{x_{0}} \wedge \tau^{-}_{0} > t]dt \label{} \\
		= &\frac{W^{(q)}(x_{0} - x)}{W^{(q)}(x_{0})} \int_{0}^{x_{0}} W^{(q)}(y)dy - \int_{0}^{x_{0}} W^{(q)}(y - x)dy \label{eq09}
	\end{align}
	Taking the limit as $q \to -\lambda$, we have $\int_{0}^{\infty}\mathrm{e}^{\lambda t} \bP_{x}[\tau^{+}_{x_{0}} \wedge \tau^{-}_{0} > t]dt \to \infty$, and this is contradiction.
	Thus, we obtain the positivity of $W^{(-\lambda)}(x)$ on $(0,\infty)$ for $0 < \lambda < \lambda_{0}$.
	Since $W^{(q)}(x)$ is continuous in $q$, the function $W^{(-\lambda_{0})}(x)$ is non-negative.
	From \eqref{eq06} and \eqref{eq11}, we obtain
	\begin{align}
		W^{(-\lambda_{0})}(x) =  \lambda W(x) \int_{0}^{\infty} \bP_{y}[\tau_{x}^{+} \leq \tau^{-}_{0}] W^{(-\lambda_{0})}(y)dy > 0 \quad  (x > 0). \label{}
	\end{align}

	% Suppose $W^{(-\lambda_{0})}(x_{0}) = 0$ for some $x_{0} > 0$.
	% We may assume $x_{0}$ is the infimum of such roots.
	% From \eqref{eq09} for $x_{0}$, we have
	% \begin{align}
	% 	\lim_{q \to -\lambda_{0}+0} \int_{0}^{\infty}\mathrm{e}^{-q t} \bP_{x}[\tau^{+}_{x_{0}} \wedge \tau^{-}_{0} > t]dt = \infty. \label{}
	% \end{align} 
	% From the monotonicity, we have for every $\delta > 0$
	% \begin{align}
	% 	\lim_{q \to -\lambda_{0}+0} \int_{0}^{\infty}\mathrm{e}^{-q t} \bP_{x}[\tau^{+}_{x_{0} + \delta} \wedge \tau^{-}_{0} > t]dt = \infty. \label{eq10}
	% \end{align}
	% Set $\delta_{0} := \sup\{ \delta \geq 0 \mid \inf_{\delta' \in [0,\delta]} \int_{0}^{x_{0} + \delta'}W^{(-\lambda_{0})}(y)dy > 0 \}$. Note that $\delta_{0} > 0$ from the positivity of $W^{(-\lambda_{0})}(x)$ on $(0,x_{0})$.
	% Then from \eqref{eq09} for $x_{0}$ being $x_{0} + \delta \ (\delta \in (0,\delta_{0}))$ and \eqref{eq10}, we have $W^{(-\lambda_{0})}(x_{0} + \delta) = 0$.
	% Thus, it holds $\delta_{0} = \infty$ and $W(x) = 0 \ (x \geq x_{0})$.
	% On the other hand, from \eqref{charEq} we have
	% \begin{align}
	% 	W(x_{0} + \delta) &= \lambda_{0} \int_{0}^{x_{0}} W(x_{0} + \delta - y) W^{(-\lambda_{0})}(y)dy \label{}
	% \end{align}

\end{proof}

If the function $W^{(-\lambda)}$ is non-negative on $(0,\infty)$, it induces a quasi-stationary distribution.

\begin{Lem} \label{lem:scaleFuncInduceQSD}
	Let $\lambda > 0$ and assume the function $W^{(-\lambda)}$ is non-negative on $(0,\infty)$.
	Then the distribution
	\begin{align}
		\nu_{\lambda}(dx) := \lambda W^{(-\lambda)}(x)dx \label{}
	\end{align}
	is a quasi-stationary distribution with $\bP_{\nu_{\lambda}}[\tau^{-}_{0} > t] = \mathrm{e}^{-\lambda t} \ (t > 0)$.
\end{Lem}

\begin{proof}
	Consider the Laplace transform of $\bP_{\nu_{\lambda}}[X_{t} \in dx, \tau^{-}_{0} > t]$.
	For $q > \lambda$, it holds from Corollary \ref{cor:potentialDensity}, \eqref{eq05} and \eqref{charEq}
	\begin{align}
		& \int_{0}^{\infty}\mathrm{e}^{- qt} \bP_{\nu_{\lambda}}[X_{t} \in dy, \tau^{-}_{0} > t]dt \label{} \\
		=&\lambda \left(\int_{0}^{\infty}W^{(-\lambda)}(x)u^{(q)}(x,y)dx\right) dy \label{} \\
		=&\lambda \left(\int_{0}^{\infty}W^{(-\lambda)}(x)(\mathrm{e}^{- x\Phi(q)}W^{(q)}(y) - W^{(q)}(y-x))dx\right)dy \label{} \\
		=& \left(\frac{\lambda W^{(q)}(y)}{\psi(\Phi(q)) + \lambda} - \lambda W^{(q)} \ast W^{(-\lambda)}(y)\right)dy \label{} \\
		=&\frac{\lambda}{q + \lambda} (W^{(q)}(y) - (q+\lambda) W^{(q)} \ast W^{(-\lambda)}(y))dy \label{} \\
		=& \frac{\nu_{\lambda}(dy)}{q+\lambda}. \label{}
	\end{align}
	Therefore, we obtain
	\begin{align}
		\bP_{\nu_{\lambda}}[X_{t} \in dx, \tau^{-}_{0} > t]dt = \mathrm{e}^{-\lambda t}\nu_{\lambda}(dx)dt, \label{}
	\end{align}
	which shows that $\nu_{\lambda}$ is a quasi-stationary distribution.
\end{proof}

We prove Theorem \ref{thm:charOfQSD}.

\begin{proof}[Proof of Theorem \ref{thm:charOfQSD}]
	It is obvious that $\lambda_{0} > 0$ is necessary for existence of a quasi-stationary distribution.
	Let $\lambda_{0} > 0$.
	From Lemmas \ref{lem:non-negativityOfScale} and \ref{lem:scaleFuncInduceQSD}, the distributions $\{\nu_{\lambda}\}_{0 < \lambda \leq \lambda_{0}}$ are quasi-stationary distributions.
	If there is another quasi-stationary distribution $\nu$, from Lemma \ref{lem:charOfQSD1} the distribution $\nu$ have a density $\lambda W^{(-\lambda)}(x)$ for some $\lambda > \lambda_{0}$.
	It is, however, impossible because from Lemma \ref{lem:scaleFuncInduceQSD} it follows $\bE_{\nu}[\mathrm{e}^{(\lambda-\eps) \tau^{-}_{0}}] < \infty$ for every $\eps > 0$ and contradicts to the definition of $\lambda_{0}$.
\end{proof}

As a corollary of Theorem \ref{thm:charOfQSD}, we show the set of quasi-stationary distributions is totally ordered by the following two stochastic orders.
For a probability distribution $\mu$ and $\nu$ on $[0,\infty)$, we say $\mu$ is smaller than $\nu$ in the \textit{Laplace transform ratio order} and denote $\mu \leq_{\mathrm{Lt-r}} \nu$ when
\begin{align}
	\frac{\int_{0}^{\infty}\mathrm{e}^{-\beta x}\nu(dx)}{\int_{0}^{\infty}\mathrm{e}^{-\beta x}\mu(dx)} \quad \text{is non-increasing for } \beta \in [0,\infty), \label{}
\end{align}
and, we say $\mu$ is smaller than $\nu$ in the \textit{reversed Laplace transform ratio order} and denote $\mu \leq_{\mathrm{r-Lt-r}} \nu$ when
\begin{align}
	\frac{1 - \int_{0}^{\infty}\mathrm{e}^{-\beta x}\nu(dx)}{1 - \int_{0}^{\infty}\mathrm{e}^{-\beta x}\mu(dx)} \quad \text{is non-increasing for } \beta \in [0,\infty), \label{}
\end{align}
See e.g., \cite[Chapter 5.B]{StochasticOrder} for basic properties of these orders and relation to other stochastic orders.

\begin{Cor} \label{cor:QSDsOrdered}
	Let $\lambda_{0} > 0$. For $0 < \lambda < \lambda' \leq \lambda_{0}$, it holds
	\begin{align}
		\nu_{\lambda'} \leq_{\mathrm{Lt-r}} \nu_{\lambda} \quad \text{and} \quad \nu_{\lambda'} \leq_{\mathrm{r-Lt-r}} \nu_{\lambda}. \label{}
	\end{align}
	In particular, the probability $\nu_{\lambda_{0}}$ is the minimal element in both orders.
\end{Cor}

\begin{proof}
	Since $W^{(-\lambda)} \ (\lambda \in (0,\lambda_{0}])$ is positive, we can analytically extend the equality \eqref{eq05} to $\beta \geq 0$.
	Then we have for $0 < \lambda < \lambda' \leq \lambda_{0}$,
	\begin{align}
		\frac{\int_{0}^{\infty}\mathrm{e}^{-\beta x}\nu_{\lambda}(dx)}{\int_{0}^{\infty}\mathrm{e}^{-\beta x}\nu_{\lambda'}(dx)} = \frac{\lambda}{\lambda'} \cdot \frac{\psi(\beta) + \lambda'}{\psi(\beta) + \lambda} \quad \text{and} \quad \frac{1 - \int_{0}^{\infty}\mathrm{e}^{-\beta x}\nu_{\lambda}(dx)}{1 - \int_{0}^{\infty}\mathrm{e}^{-\beta x}\nu_{\lambda'}(dx)} = \frac{\psi(\beta) + \lambda'}{\psi(\beta) + \lambda}. \label{}
	\end{align}
	Since it holds
	\begin{align}
		\frac{d}{d\beta}\left( \frac{\psi(\beta) + \lambda'}{\psi(\beta) + \lambda} \right)
		= \frac{\psi'(\beta)(\lambda - \lambda')}{(\psi(\beta) + \lambda)^{2}} < 0, \label{}
	\end{align}
	the proof is complete.
\end{proof}

We prove the first equality of \eqref{eq07} in Theorem \ref{thm:spectralBottom}.

\begin{proof}[Proof of Theorem \ref{thm:spectralBottom} (the first half)]
	Set
	\begin{align}
		\tilde{\lambda}_{0} := \sup \left\{ \lambda \geq 0 ~\middle|~ W^{(-\lambda)}(x) > 0 \quad \text{for every $x > 0$} \right\}. \label{} 
	\end{align}
	From Lemma \ref{lem:non-negativityOfScale}, we have $\lambda_{0} \leq \tilde{\lambda_{0}}$.
	For every $0 < \lambda < \tilde{\lambda_{0}}$, it holds from Lemma \ref{lem:scaleFuncInduceQSD} that $\nu_{\lambda}$ is a quasi-stationary distribution with $\bP_{\nu_{\lambda}}[\tau^{-}_{0} > t] = \mathrm{e}^{-\lambda t}$, which implies $\bE_{x}[\mathrm{e}^{(\lambda - \eps) \tau^{-}_{0}}] < \infty$ for every $\eps > 0$ and $x > 0$.
	Thus, it follows $\lambda \leq \lambda_{0}$.
\end{proof}

To show the second equality of \eqref{eq07}, we need some more preparation.
Since the function $\Phi$ satisfies \eqref{eq12}, we may extend the domain of $\Phi$ to $(-\lambda_{0},\infty)$ by defining
\begin{align}
	\Phi(q) := -\log \bE_{1}[\mathrm{e}^{-q\tau^{-}_{0}}, \tau_{0}^{-} < \infty]. \label{eq23}
\end{align}
Note that $\Phi$ is increasing and analytic on the interval.

When $\lambda_{0} > 0$, the exponential moments of the process $X$ exist.

\begin{Prop} \label{prop:ExponentialIntegrabilityOfFTImpliesFinitenessOfMGF}
	Suppose $\lambda_{0} > 0$.
	Then for $0 <  \theta < - \Phi(-\lambda_{0}+) := -\lim_{\lambda \to \lambda_{0}-} \Phi(-\lambda)$, it holds $|\psi(-\theta)| < \infty$.
\end{Prop}

For the proof, we need the following lemma.
We owe the idea of the proof to \cite[Corollary 8.9]{KyprianouText}.

\begin{Lem} \label{lem:propOfWPhi}
	Suppose $\lambda_{0} > 0$.
	Then for $\lambda \in (0,\lambda_{0})$, the following hold:
	\begin{enumerate}
		\item The function $W_{\Phi(-\lambda)}(x) := \mathrm{e}^{-\Phi(-\lambda)x}W^{(-\lambda)}(x)$ is strictly increasing.
		\item For $r < \lambda$, it holds
		\begin{align}
			\int_{0}^{\infty}\mathrm{e}^{-\Phi(-r) x}W^{(-\lambda)}(x)dx = \frac{1}{\lambda-r}. \label{eq14}
		\end{align}
		\item $\lim_{x \to \infty} W_{\Phi(-\lambda)}(x) = \Phi^{+}(-\lambda)$.
	\end{enumerate}
\end{Lem}

\begin{proof}
	(i) By the analytic extension of \eqref{eq12} and \eqref{eq01}, it holds for $x,y > 0$
	\begin{align}
		\frac{\mathrm{e}^{-\Phi(-\lambda)x}W^{(-\lambda)}(x)}{\mathrm{e}^{-\Phi(-\lambda)(x+y)}W^{(-\lambda)}(x+y)} = \frac{\bE_{y}[\mathrm{e}^{\lambda\tau^{-}_{0}}, \tau^{+}_{x+y} > \tau^{-}_{0}]}{\bE_{y}[\mathrm{e}^{\lambda\tau^{-}_{0}}, \tau_{0}^{-} < \infty]}, \label{eq15}
	\end{align}
	which is smaller than one.
	Thus, the function $\mathrm{e}^{-\Phi(-\lambda)x}W^{(-\lambda)}(x)$ is strictly increasing.

	(ii) Taking the limit in \eqref{eq15} as $x \to \infty$, it converges to $1$, which implies that the function $\mathrm{e}^{-\Phi(-\lambda)\log x}W^{(-\lambda)}(\log x) \ (x \geq 1)$ is slowly varying at $\infty$.  Then it follows for every $\eps > 0$ that $\lim_{x \to \infty}\mathrm{e}^{-\eps x}W_{\Phi(-\lambda)}(x) = 0$  (see e.g., \cite[Proposition 1.3.6]{Regularvariation}). Thus, the integral in the LHS of \eqref{eq14} is finite.
	Both sides of \eqref{eq14} are analytic on the interval $(-\lambda,\infty)$ as the functions of $-r$ and coincide for $-r \in (\lambda,\infty)$ by \eqref{eq05}.
	Hence, from the identity theorem the equality holds for $r < \lambda$.

	(iii) From (i) and (ii), we have
	\begin{align}
		W_{\Phi(-\lambda)}(\infty) &= \lim_{\beta \to 0+} \beta \int_{0}^{\infty} \mathrm{e}^{-\beta x}W_{\Phi(-\lambda)}(x)dx \label{} \\
		&= \lim_{r \to \lambda-}(\Phi(-r) - \Phi(-\lambda))\int_{0}^{\infty}\mathrm{e}^{-\Phi(-r) x}W^{(-\lambda)}(x)dx \label{} \\
		&= \lim_{r \to \lambda-}\frac{\Phi(-r) - \Phi(-\lambda)}{\lambda-r} \label{} \\
		&= \Phi^{+}(-\lambda). \label{}
	\end{align}
\end{proof}

\begin{Rem}
	As we will see in Remark \ref{rem:extOfPhi}, it holds $\Phi(-\lambda_{0}) = -\bE_{1}[\mathrm{e}^{\lambda_{0}\tau^{-}_{0}}, \tau_{0}^{-} < \infty] > -\infty$.
	We may easily check that the results in Lemma \ref{lem:propOfWPhi} also hold for $\lambda = \lambda_{0}$ by the same argument.
\end{Rem}

\begin{proof}[Proof of Proposition \ref{prop:ExponentialIntegrabilityOfFTImpliesFinitenessOfMGF}]
	Let $0 < \theta < -\Phi(-\lambda_{0}+)$. Take $\lambda \in (0,\lambda_{0})$ such that $0 < \theta < -\Phi(-\lambda)$.
	By the analytic extension of Corollary \ref{cor:potentialDensity}, Lemma \ref{lem:propOfWPhi} and the monotone convergence theorem, we have for $x,a > 0$
		\begin{align}
		&\int_{0}^{\infty} \mathrm{e}^{\lambda t} \bE[\mathrm{e}^{\theta X_{t}},0 \leq X_{t} \leq a, \tau^{-}_{-x} > t]dt \label{} \\
		= &\int_{0}^{a} (\mathrm{e}^{-\Phi(-\lambda)x}W^{(-\lambda)}(y+x) - W^{(-\lambda)}(y))\mathrm{e}^{\theta y}dy \label{} \\
		\xrightarrow{x \to \infty}& \int_{0}^{a} (\Phi'(-\lambda)\mathrm{e}^{\Phi(-\lambda)y} - W^{(-\lambda)}(y))\mathrm{e}^{\theta y}dy \label{} \\
		\xrightarrow{a \to \infty} &\int_{0}^{\infty} (\Phi'(-\lambda)\mathrm{e}^{\Phi(-\lambda)y} - W^{(-\lambda)}(y))\mathrm{e}^{\theta y}dy < \infty. \label{} 
	\end{align}
	Thus, we obtain
	\begin{align}
		\int_{0}^{\infty} \mathrm{e}^{\lambda t} \bE[\mathrm{e}^{\theta X_{t}}, X_{t} \geq 0]dt 
		=\int_{0}^{\infty} (\Phi'(-\lambda)\mathrm{e}^{\Phi(-\lambda)y} - W^{(-\lambda)}(y))\mathrm{e}^{\theta y}dy < \infty. \label{}
	\end{align}
\end{proof}

We prove the second equality of \eqref{eq07} in Theorem \ref{thm:spectralBottom}.

\begin{proof}[Proof of Theorem \ref{thm:spectralBottom} (the latter half)]
	At first, we show 
	\begin{align}
		-\psi(-\theta_{0}) \leq \lambda_{0}. \label{eq30}
	\end{align}	
	Since $\bE[\mathrm{e}^{\theta_{0} X_{1}}] < \infty$, we may consider the Esscher transform $\bP^{\theta_{0}}$ defined in \eqref{eq16}.
	Note that $\psi_{\theta_{0}}^{+}(0) = \psi^{+}(-\theta_{0}) \geq 0$.
	We denote the $q$-scale function of the transformed process by $W^{(q)}_{\theta_{0}}$.
	From \cite[Lemma 8.4]{KyprianouText}, we have
	\begin{align}
		W^{(\psi(-\theta_{0}))}(x) = \mathrm{e}^{-\theta_{0} x}W_{\theta_{0}}^{(0)}(x), \label{}
	\end{align}
	from which we see the function $W^{(\psi(-\theta_{0}))}$ is positive on $(0,\infty)$.
	Hence, from the first equality of \eqref{eq07} we obtain \eqref{eq30}.

	When $\theta_{0} > 0$, it holds $-\psi(-\theta_{0}) > 0$ from the definition of $\theta_{0}$. 
	Thus, it follows $\lambda_{0} > 0$ from \eqref{eq30}.
	When $\lambda_{0} > 0$, it follows from \eqref{eq31} that $\psi^{+}(0) = 1 / \Phi(0)' > 0$.
	Thus, from Proposition \ref{prop:ExponentialIntegrabilityOfFTImpliesFinitenessOfMGF}
	we have $\theta_{0} > 0$.
	Therefore, we obtain the equivalence $\lambda_{0} > 0 \ \Leftrightarrow \ \theta_{0} > 0$
	and the desired equality when $\theta_{0} = \lambda_{0} = 0$.

	To complete the proof, it is enough to show $-\psi(-\theta_{0}) \geq \lambda_{0}$ when $\theta_{0} > 0$.
	For $0 < \theta < \theta_{0}$, it is not difficult to see that the right inverse $\Phi_{\theta}$ of the Esscher transformed $\bP^{\theta}$ satisfies
	\begin{align}
		\Phi_{\theta}(q) = \Phi(q + \psi(-\theta)) + \theta \quad (q \geq 0). \label{}
	\end{align}
	Since $\psi_{\theta}^{+}(0) = \psi'(\theta) > 0$,
	we have $0 = \Phi_{\theta}(0) = \Phi(\psi(-\theta)) + \theta$.
	Thus, it holds
	\begin{align}
		-\theta = \Phi(\psi(-\theta)) \quad (0 < \theta \leq \theta_{0}). \label{eq24}
	\end{align}
	Suppose $-\psi(-\theta_{0}) < \lambda_{0}$.
	Then it follows $\theta_{0} < -\Phi(-\lambda_{0}+)$ from \eqref{eq24}.
	From Proposition \ref{prop:ExponentialIntegrabilityOfFTImpliesFinitenessOfMGF},
	for $\theta_{0} < \tilde{\theta} < -\Phi(-\lambda_{0}+)$ it holds $|\psi(-\tilde{\theta})| < \infty$.
	By considering the Esscher transform $\bP^{\tilde{\theta}}$,
	from the definition of $\theta_{0}$ we see $\psi'(-\tilde{\theta}) < 0$,
	and it implies $\psi'(-\theta_{0}) = 0$.
	Again from \eqref{eq24} we have $\Phi'(\psi(-\theta)) = 1/ \psi'(-\theta)$ for $\theta < \theta_{0}$, and therefore $\lim_{\theta \to \theta_{0}-}\Phi'(\psi(-\theta)) = \infty$.
	Since the function $\Phi$ is defined on the interval $(-\lambda_{0},\infty)$ and strictly concave,
	it follows $-\lambda_{0} \geq \psi(-\theta_{0})$, which contradicts to the assumption $-\psi(-\theta_{0}) < \lambda_{0}$. Hence, we obtain $\psi(-\theta_{0}) = -\lambda_{0}$.
	% Since the Laplace exponent $\psi(\beta)$ is finite and increasing on $(-\Phi(-\lambda_{0}+0),\infty)$ from Proposition \ref{prop:ExponentialIntegrabilityOfFTImpliesFinitenessOfMGF},
	% it follows $\theta_{0} \geq -\Phi(-\lambda_{0}+0)$.
	% Take $\lambda < \lambda_{0}$ arbitrarily, and take $\beta > -\theta_{0}$ so that $-\beta \geq -\Phi(-\lambda)$.
	% It is enough to show $\psi(\Phi(-\lambda)) = -\lambda$.
	% Consider the Esscher transform for $\theta = -\Phi(-\lambda)$.
	% It is not difficult to see the right inverse $\Phi_{\theta}$ of the transformed process satisfies
	% \begin{align}
	% 	\Phi_{\theta}(q) = \Phi(q + \psi(-\theta)) + \theta \quad (q \geq 0). \label{}
	% \end{align}
	% Since $\psi_{\theta}'(0+) = \psi'(\theta) > 0$,
	% we have $0 = \Phi_{\theta}(0) = \Phi(\psi(-\theta)) + \theta$, that is, $-\theta = \Phi(\psi(-\theta))$.
	% Thus, we have $\Phi(-\lambda) = \Phi(\psi(\Phi(-\lambda)))$.
	% Since the function $\Phi$ is increasing on $(-\lambda_{0},\infty)$, it holds $-\lambda = \psi(\Phi(-\lambda))$.
\end{proof}

\begin{Rem} \label{rem:extOfPhi}
	From \eqref{eq24} and the monotone convergence theorem, we see that $\Phi(-\lambda_{0}) = -\theta_{0} > -\infty$, that is, $\bE_{1}[\mathrm{e}^{\lambda_{0}\tau^{-}_{0}}, \tau_{0}^{-} < \infty] = \mathrm{e}^{\theta_{0}} < \infty$.
	Thus, the process $X$ is \textit{$\lambda_{0}$-transient} in the terminology of Tuominen and Tweedie \cite{TuominenTweedie}.
\end{Rem}

We prove Corollary \ref{cor:YaglomLimit}.
For the definition of \textit{class A} and \textit{class B}, see \cite[Definition 1,2]{KyprianouPalmowski}.

\begin{proof}[Proof of Corollary \ref{cor:YaglomLimit}]
	Under the assumption (i) or (ii), the process $X$ belongs to class A or class B, respectively.
	From \cite[Theorem 1]{KyprianouPalmowski}, the convergence \eqref{eq28} holds for every one-point initial distribution on $(0,\infty)$.
	Since $\bE_{x}[\mathrm{e}^{\theta X_{t}}, \tau^{-}_{0} > t] \ (\theta \in [0,\theta_{0}), \ t > 0)$ is increasing in $x > 0$, we see for every $0 < r \leq x \leq R$ that
	\begin{align}
		\bE_{x}[\mathrm{e}^{\theta X_{t}} \mid \tau^{-}_{0} > t] 
		\leq \frac{\bP_{R}[\tau^{-}_{0} > t]}{\bP_{r}[\tau^{-}_{0} > t]} \cdot \bE_{R}[\mathrm{e}^{\theta X_{t}} \mid \tau^{-}_{0} > t]. \label{}
	\end{align}
	Since the RHS is bounded in $t \geq 0$ from \cite[Lemma 3]{KyprianouPalmowski},
	we obtain the desired result from the dominated convergence theorem.
\end{proof}

\section{Examples} \label{section:example}

Here we see some examples such that we can give $\lambda_{0}$ in \eqref{eq08}, $\theta_{0}$ in \eqref{thetaZero} and the scale function $W^{(q)}$ explicitly to some extent.
From these, we can compute the set of quasi-stationary distributions by the relation $\nu_{\lambda}(dx) = \lambda W^{(-\lambda)}(x)dx \ (\lambda \in (0,\lambda_{0}])$.

\subsection*{Brownian motion with drift}

Let $X_{t} := -\mu t + \sigma B_{t}$, where $\mu > 0$ and $B$ is a standard Brownian motion.
Existence and domain of attraction of quasi-stationary distributions of this process have already been studied in \cite{DoAofBM}.
Its Laplace exponent and right inverse are
\begin{align}
	\psi(\beta) = \mu \beta + \frac{1}{2}\sigma^{2} \beta^{2} \quad \text{and} \quad \Phi(q) = \frac{1}{\sigma^{2}}(\sqrt{\mu^{2} + 2q\sigma^{2}} - \mu).
\end{align}
We easily see $\theta_{0} = \mu/ \sigma^{2}$ and $\lambda_{0} = -\psi(-\theta_{0}) = \mu^{2} / (2\sigma^{2})$.
It holds
\begin{align}
	\frac{1}{\psi(\beta) - q} = \frac{1}{\sqrt{\mu^{2} + 2q\sigma^{2}}} \left( \frac{1}{\beta - \Phi(q)} - \frac{1}{\beta - \alpha(q)} \right), \label{eq20}
\end{align}
where $\alpha(q)$ is the smaller solution of the quadratic equation $\psi(\beta) - q = 0$:
\begin{align}
	\alpha(q) = \frac{- \mu - \sqrt{\mu^{2} + 2q \sigma^{2}}}{\sigma^{2}}. \label{}
\end{align}
Thus, from \eqref{eq20} we have the scale functions 
\begin{align}
	W^{(q)}(x) 
	= &\left\{
	\begin{aligned}
		&\frac{2\mathrm{e}^{-\mu x / \sigma^{2}}}{\sqrt{\mu^{2} + 2q\sigma^{2} }} \sinh (x \sqrt{\mu^{2} + 2q\sigma^{2}} / \sigma^{2}) & (q \neq - \mu^{2}/(2\sigma^{2})) \\
		&(2x / \sigma^{2})\mathrm{e}^{-\mu x / \sigma^{2}} & (q = - \mu^{2}/(2\sigma^{2}) )
	\end{aligned}
	\right.
	. \label{}
\end{align}

% The quasi-stationary distributions are
% \begin{align}
% 	\nu_{\lambda}(dx) = \left\{
% 	\begin{aligned}
% 		& \frac{2 \lambda \mathrm{e}^{-\mu x/ \sigma^{2}}}{\sqrt{\mu^{2} - 2\lambda \sigma^{2} }} \sinh (x \sqrt{\mu^{2} - 2 \lambda \sigma^{2}} / \sigma^{2}) dx & (0  < \lambda < \mu^{2} / (2\sigma^{2})), \\
% 		& \frac{\mu^{2}x}{\sigma^{4}} \mathrm{e}^{-\mu x / \sigma^{2}}dx  & (\lambda = \mu^{2} / (2\sigma^{2})).
% 	\end{aligned}
% 	\right.
% 	\label{}
% \end{align}

\subsection*{Compound Poisson process with exponentially distributed jumps and drift}

Let $X_{t}$ be a spectrally positive L\'evy process with the Laplace exponent
\begin{align}
	\psi(\beta) = \mu \beta - c\int_{0}^{\infty}(1 - \mathrm{e}^{-\beta x})\rho \mathrm{e}^{-\rho x}dx = \mu \beta - \frac{c\beta}{\beta + \rho} \label{}
\end{align}
for $\mu,c,\rho > 0$, that is, the process $X$ is a compound Poisson process with the drift $-\mu$ and the L\'evy measure $\Pi(dx) = c\rho \mathrm{e}^{-\rho x}1\{x > 0\}dx$.
To ensure existence of quasi-stationary distributions, we suppose $\psi'(0) = \mu - c / \rho > 0$.
We easily see $\theta_{0} = \rho - \sqrt{c\rho / \mu}$ and $\lambda_{0} = (\sqrt{\mu\rho} - \sqrt{c})^{2}$.
The right inverse of $\psi$ is
\begin{align}
	\Phi(q) = \frac{-(\mu \rho - c - q) + \sqrt{D(q)}}{2\mu} \quad (q \geq 0), \label{}
\end{align}
for $D(q) := (\mu \rho - c - q)^{2} + 4\mu \rho q$.
For $q \in \bC$ and $\beta > \Phi(|q|)$, it holds
\begin{align}
	\frac{1}{\psi(\beta) - q} &= \frac{\beta + \rho}{\mu \beta^{2} + (\mu \rho - c - q)\beta - q\rho} \label{} \\
	&= \frac{\beta + \rho}{\mu(\beta - \Phi(q))(\beta - \alpha(q))} \label{} \\
	&= \frac{1}{\sqrt{D(q)}} \left( \frac{\Phi(q) + \rho}{\beta - \Phi(q)} - \frac{\alpha(q) + \rho}{\beta - \alpha(q)}\right) \label{} \\
	&= \frac{1}{\sqrt{D(q)}} \int_{0}^{\infty} \mathrm{e}^{-\beta x}( (\Phi(q) + \rho)\mathrm{e}^{ \Phi(q) x } - (\alpha(q) + \rho)\mathrm{e}^{ \alpha(q) x}) dx \label{eq21}
\end{align}
for
\begin{align}
	\alpha(q) := \frac{-(\mu \rho - c - q) - \sqrt{D(q)}}{2\mu}, \label{}
\end{align}
where we consider the limit as $q \to q_{0}$ when $D(q_{0}) = 0$.
Set $\gamma(q) := \mu\rho - c - q$.
Since $D(q) = 0 \Leftrightarrow q = - (\sqrt{\mu\rho} \pm \sqrt{c})^{2}$,
from \eqref{eq21} the scale functions are for $q \neq - (\sqrt{\mu\rho} \pm \sqrt{c})^{2}$
\begin{align}
	W^{(q)}(x)
	&= \frac{ (\Phi(q) + \rho)\mathrm{e}^{ \Phi(q) x } - (\alpha(q) + \rho)\mathrm{e}^{ \alpha(q) x}}{\sqrt{D(q)}} \label{} \\
	&= \mathrm{e}^{-\frac{\gamma(q)}{2\mu}x} \left( \frac{-\gamma(q) / \mu + 2\rho}{\sqrt{D(q)}} \sinh \left( \frac{\sqrt{D(q)}}{2\mu}x \right) + \frac{1}{\mu} \cosh \left( \frac{\sqrt{D(q)}}{2\mu} x \right) \right) \label{}
\end{align}
and for $q = - (\sqrt{\mu\rho} \pm \sqrt{c})^{2}$
\begin{align}
	W^{(q)}(x) = \mathrm{e}^{-\rho(1 \pm \sqrt{c/(\mu\rho)})x} \left( \mp\frac{\rho}{\mu}\sqrt{\frac{c}{\mu\rho}}x + \frac{1}{\mu} \right). \label{}
\end{align}

\subsection*{Meromorphic L\'evy process}

We start from considering the spectrally positive L\'evy process whose L\'evy measure $\Pi$ has a completely monotone density, that is, suppose it holds for a Radon measure $\alpha$ on $(0,\infty)$
\begin{align}
	\int_{0}^{\infty}\mathrm{e}^{-sx}\alpha(ds) < \infty  \quad \text{and} \quad 
	\Pi(dx) = \left(\int_{0}^{\infty}\mathrm{e}^{-sx}\alpha(ds)\right)dx \quad (x > 0). \label{}
\end{align}
It is not difficult to see that for $\Pi$ to be a L\'evy measure it is necessary and sufficient that $\int_{0}^{\infty}s^{-1}(s^{2} + 1)^{-1} \alpha(dx) < \infty$. 
We also easily see the integrability $\int_{1}^{\infty}\mathrm{e}^{\delta x}\Pi(dx) < \infty$ for some $\delta > 0$ is equivalent to the infimum $r$ of the support of $\alpha$ is positive, that is,
there exists $r > 0$ such that $\alpha(0,r) = 0$ and $\alpha(0,r+\eps) > 0$ for every $\eps > 0$.
Then the Laplace exponent is given by
\begin{align}
	\psi(\beta) = -a \beta + \frac{1}{2}\sigma^{2}\beta^{2} 
	+ \int_{r}^{\infty}\left( \frac{1}{\beta + s} - \frac{1}{s} + \frac{\beta}{s^{2}} \right)\nu(ds) - \beta \int_{r}^{\infty}\frac{\mathrm{e}^{-s}(s+1)}{s^{2}}\alpha(ds) \label{}
\end{align}
for a constant $a \in \bR$ and a Gaussian coefficient $\sigma \geq 0$.
Its derivative is
\begin{align}
	\psi'(\beta) = -a + \sigma^{2}\beta - \int_{r}^{\infty}\left( \frac{1}{(\beta + s)^{2}} - \frac{1}{s^{2}} \right) \alpha(ds) - \int_{r}^{\infty}\frac{\mathrm{e}^{-s}(s+1)}{s^{2}}\alpha(ds). \label{}
\end{align}
We assume
\begin{align}
	\psi'(0) = -a - \int_{r}^{\infty}\frac{\mathrm{e}^{-s}(s+1)}{s^{2}}\alpha(ds) > 0. \label{}
\end{align}
Obviously, it holds $|\psi(\beta)| < \infty$ for $\beta > -r$ and $\psi(-r+) = \infty$.
Since $\psi(\beta)$ is convex on $(-r,\infty)$, there exists a unique root of $\psi'(\beta)$ on $(-r,0)$, and it is the minus of $\theta_{0}$ and $\lambda_{0} = -\psi(-\theta_{0})$.

To compute the scale function, we further assume the process $X$ is in the \textit{meromorphic class} (see e.g., \cite{MeromorphicLevy} and \cite[Section 6.5.4]{KyprianouText} for details), that is, the measure $\alpha$ is of the form
\begin{align}
	\alpha(ds) = \sum_{i \geq 1} a_{i}\rho_{i}\delta_{\rho_{i}}(ds) \label{eq17}
\end{align}
for $a_{i}, \rho_{i} > 0$, where the sequence $\{ \rho_{i} \}_{i \geq 1}$ is increasing and $\lim_{i \to \infty}\rho_{i} = \infty$. For $\Pi$ to be a L\'evy measure we suppose $\sum_{i \geq 1}a_{i}\rho_{i}^{-2} < \infty$.
By setting $\rho_{n} = \infty$ for some $n \geq 1$  in the argument below, we may treat the case when the range of summation in \eqref{eq17} is finite.

The Laplace exponent is
\begin{align}
	\psi(\beta) = -a \beta + \frac{1}{2}\sigma^{2}\beta^{2} 
	+ \sum_{i \geq 1}a_{i}\rho_{i}\mathrm{e}^{-\rho_{i}}\left( \frac{1}{\beta + \rho_{i}} - \frac{1}{\rho_{i}} + \frac{\beta}{\rho_{i}^{2}} \right) - \beta \sum_{i \geq 1}a_{i} \frac{\mathrm{e}^{-2\rho_{i}}(\rho_{i}+1)}{\rho_{i}}. \label{}
\end{align}
For $q > -\lambda_{0}$, the function $\psi(\beta) - q$ has simple poles at $\beta = -\rho_{i} \ (i \geq 1)$ and has simple roots at $\beta = \Phi(q)$ and $-\zeta_{i}(q) \ (i \geq 1)$, where $\zeta_{i}(q)$ is a positive number such that $\rho_{i} < \zeta_{i+1}(q) < \rho_{i+1}$ for $i \geq 0$, where we denote $\rho_{0} := \Phi(q)$.
Then by the factorization theorem, it holds for $\beta > \Phi(|q|)$
\begin{align}
	\int_{0}^{\infty}\mathrm{e}^{-\beta x} W^{(q)}(x)dx &= - \frac{1}{q}\cdot \frac{1}{1 - \beta / \Phi(q)} \prod_{i \geq 1}\frac{1 + \beta / \rho_{i}}{1 + \beta / \zeta_{i}(q)} \label{} \\
	&= \frac{1}{\psi'(\Phi(q))(\Phi(q) - \beta)} + \sum_{i \geq 1} \frac{1}{\psi'(-\zeta_{i}(q))(\zeta_{i}(q) + \beta)}. \label{}
\end{align}
Thus, we have
\begin{align}
	W^{(q)}(x) = \frac{\mathrm{e}^{\Phi(q)x}}{\psi'(\Phi(q))} + \sum_{i \geq 1}\frac{\mathrm{e}^{-\zeta_{i}(q)x}}{\psi'(-\zeta_{i}(q))}. \label{eq18}
\end{align}
Since $\psi$ is decreasing on $(-\rho_{1},-\theta_{0})$, it holds $\zeta_{1}(q) \to \theta_{0}$ as $q \to -\lambda_{0}+$.
Hence, in the case $q = -\lambda_{0}$ the order of the root $\beta = \Phi(-\lambda_{0})$ of $\psi(\beta) + \lambda_{0}$ is two.
The argument for $q > -\lambda_{0}$ also works by an obvious modification, and we obtain
\begin{align}
	W^{(-\lambda_{0})}(x) = \frac{2x\mathrm{e}^{\Phi(-\lambda_{0})x}}{\psi''(\Phi(-\lambda_{0}))} - \frac{2\psi'''(\Phi(-\lambda_{0})) \mathrm{e}^{\Phi(-\lambda_{0})x}}{3 \psi''(\Phi(-\lambda_{0}))^{2}} + \sum_{i \geq 2}\frac{\mathrm{e}^{-\zeta_{i}(-\lambda_{0})x}}{\psi'(-\zeta_{i}(-\lambda_{0}))}. \label{eq19}
\end{align}

\bibliography{arxiv02.bbl}
\bibliographystyle{plain}

\end{document}